\documentclass[10pt]{article}

\usepackage[latin1]{inputenc}
\usepackage[T1]{fontenc}
\usepackage{amsmath,amsfonts,amssymb,amsthm,url,geometry,mathrsfs,color,esint,mathabx}
\usepackage{graphicx,tikz}
\usepackage{multind,todo,bm}
\newcommand{\e}{\varepsilon}
\newcommand{\vphi}{\varphi}

\newcommand{\st}{\ : \ }

\DeclareMathOperator{\tr}{Tr}

\DeclareMathOperator{\Id}{I}

\DeclareMathOperator{\rank}{rank}

\renewcommand{\leq}{\leqslant}

\renewcommand{\geq}{\geqslant}

\newcommand{\cD}{\mathrm{D}}

\newcommand{\mL}{\mathcal{L}}

\newcommand{\PSD}{\mathcal{PSD}}

\newcommand{\LP}{\cP(\mathcal{L}_m)}

\newcommand{\cP}{\bm{P}}

\newcommand{\R}{\mathbb{R}}
\newcommand{\C}{\mathbb{C}}

\newcommand{\cM}{\mathsf{M}}

\newcommand{\gU}{\mathsf{U}}
\newcommand{\gO}{\mathsf{O}}

\newcommand{\gSO}{\mathsf{SO}}

\newcommand{\gGL}{\mathsf{GL}}

\newcommand{\mC}{\mathcal{C}}

\newcommand{\scalar}[2]{\langle #1 , #2\rangle}
\newcommand{\braket}[2]{\langle #1 | #2\rangle}
\newcommand{\ketbra}[2]{| #1 \rangle \langle #2 |}
\newcommand{\bra}[1]{\langle #1 |}
\newcommand{\ket}[1]{| #1 \rangle}
\newcommand{\sa}{\textnormal{sa}}

\newcommand{\HS}{\textnormal{HS}}

\theoremstyle{plain}
\newtheorem{theorem}{Theorem}

\newtheorem{proposition}[theorem]{Proposition}
\newtheorem{lemma}[theorem]{Lemma}

\theoremstyle{definition}

\theoremstyle{remark}
\newtheorem{remark}[theorem]{Remark}

\allowdisplaybreaks

\title{Two proofs of St{\o}rmer's theorem} 
\author{Guillaume Aubrun\footnote{Institut
Camille Jordan, Universit\'{e} Claude Bernard Lyon 1, {aubrun@math.univ-lyon1.fr}. Supported in part by Agence Nationale de la Recherche (France) 
grants OSQPI (2011-BS01-008-02) and StoQ (2014-CE25-0003).} \ \ and \ Stanis{\l}aw J. Szarek\footnote{Case Western Reserve University {\sl and} Universit\'{e}  Pierre et Marie Curie, szarek@cwru.edu. Supported in part by grants from the National Science
Foundation (U.S.A.) and by the first ANR grant listed under GA.}}

\date{}

\begin{document}

\maketitle

\begin{abstract}
The structure of the set of positivity-preserving maps between matrix algebras is notoriously difficult to describe. The notable exceptions are the results by St{\o}rmer and Woronowicz from 1960s and 1970s settling the low dimensional cases. By duality, these results are equivalent to the Peres--Horodecki positive partial transpose criterion being able to unambiguously establish whether a state in a $2 \times 2$ or $2 \times 3$ quantum system is entangled or separable.  
However, even in these low dimensional cases, the existing arguments (known to the authors) were based on long and seemingly {\sl ad hoc} computations. We present a simple proof -- based on Brouwer's fixed point theorem -- 
for the  $2 \times 2$ case (St{\o}rmer's theorem). 
For completeness, we also include another argument 
(following the classical outline, but highly streamlined)  
based on a characterization of extreme self-maps of the Lorentz cone and on a link -- 
noticed by R. Hildebrand -- to  the $S$-lemma, 
a well-known fact from control theory and quadratic/semi-definite programming.  
\end{abstract}


\bigskip 
Denote by $\cM_n$ the space of  $n \times n$ complex matrices, 
by $\cM_n^{sa}$ the real-linear subspace of $n \times n$ {\em Hermitian} matrices, 
and by $\PSD = \PSD(\C^n)$ the cone of  {\em positive semi-definite} matrices. 
Further, let $\cP = \cP(\C^n)$ denote the cone of {\em positivity-preserving} maps  $\Phi : \cM_n^{sa} \to \cM_n^{sa}$,  
 i.e., linear maps verifying $\Phi(\PSD) \subset \PSD$. 
 
\smallskip 
In this note we will present a short proof of the following 1963 result of St{\o}rmer \cite{Stormer63}.
 
 \begin{theorem} [St{\o}rmer's theorem] \label{theorem:stormer}
A map $\Phi : \cM_2^{sa} \to \cM_2^{sa}$ belongs to $\Phi \in \cP(\C^2)$ if and only if 
\begin{equation} \label{eq:stormer}
\Phi(\rho)  = \sum_j A_j\rho A_j^\dagger+\sum_k B_k\rho^T B_k^\dagger
\end{equation}
 for some $\{A_j,B_k\}\subset\cM_2$, where $\rho^T$ denotes the transpose of $\rho$. Moreover, the total number of terms required in \eqref{eq:stormer} does not exceed $4$. 
\end{theorem}  
In what follows we will describe -- for completeness -- the background of the result and go over the (rather standard) notation.  
However,  a reader familiar with the subject may just consult the statements of Proposition \ref{prop:general-to-unital-TP} 
and Lemma \ref{lemma:interior}, and read the proof of Proposition \ref{prop:general-to-unital-TP}, which together 
take less than a page. {We point that although many proofs of St{\o}rmer's theorem appeared in the literature \cite{KCKL00,VDM01,LMO06,KSW09}, 
we are not aware of an argument along the ideas of Section \ref{section:brouwer}}.

\section{The background} 

Since maps of the form 
\begin{equation}  \label{eq:PhiM}
\Phi_M(\rho) := M\rho M^\dagger
\end{equation}
 generate the cone of {\em completely positive} maps, an equivalent restatement of Theorem \ref{theorem:stormer} is that every positivity-preserving map on $\cM_2^{sa}$ is 
{\em decomposable}, i.e., 
can be represented as a sum of a completely positive map and a {\em co-completely positive} map 
(that is, the composition of a completely positive map and the transposition).  
By duality, St{\o}rmer's theorem implies that a state on a $2 \times 2$ quantum system is separable 
iff its {\em partial transpose} is positive \cite{HHH96}. 
However, these concepts and facts are not needed for our proof; we refer the interested reader to 
the forthcoming book \cite{book}, which also contains a version of the present argument. 

The starting point of most proofs of Theorem \ref{theorem:stormer} is the realization that 
the representation \eqref{eq:stormer} is rather easy to obtain 
-- modulo {\em very} classical facts --  if the map $\Phi$ is {\em bistochastic}, that is, {\em unital}  (i.e., $\Phi(\Id) = \Id$) 
and {\em trace  preserving} (i.e., $\tr \Phi(\rho) = \tr \rho$ for all $\rho$ in the domain). 
This is because the set of states on $\C^2$ 
\[
\cD =  \cD(\C^2) = \{\rho \in \PSD(\C^2) : \tr \rho =1 \}
 \]
has a particularly simple structure:  it is a $3$-dimensional {\em Euclidean ball} (the Bloch ball). 
 More precisely, it is a ball of radius 
 $1/\sqrt{2}$ in the Frobenius (or Hilbert--Schmidt) norm and centered at $\Id/2 =: \rho_*$ (the maximally mixed state). 
Now, $\Phi$ being trace- and positivity-preserving is equivalent to $\Phi(\cD) \subset \cD$ and 
$\Phi$ being unital is equivalent to $\Phi(\rho_*) = \rho_*$, so such $\Phi$ can be identified 
with a {\em linear} operator $R : \R^3 \to \R^3$ which maps the unit ball in $\R^3$ into itself. 
This means that the norm of $R$  (the usual, operator or spectral norm) is at most $1$ and, 
consequently, $R$ is a convex combination of (at most $4$, by Carath\'eodory's theorem 
applied to the cube $[-1,1]^3$)
 isometries of $\R^3$, or elements of $\gO(3)$. 
If $R \in \gSO(3)$, then it is well-known that $R$ corresponds in the above way to the map that is 
of the form $\Phi_U$, for some $U \in \gU(2)$; this is an instance of the so-called {\em spinor map. }  
Since, as is easy to check,  
the transpose map $T$ on $\cM_2^{sa}$ corresponds to a reflection on $\R^3$, any 
$R \in \gO(3) \setminus \gSO(3)$ corresponds to a map of the form $\Phi_U \circ T$. 
Combining these observations we conclude that any bistochastic map $\Phi : \cM_2^{sa} \to \cM_2^{sa}$  
can be represented as a convex combination of at most $4$ maps that are of the form $\Phi_U$
or  $\Phi_U \circ T$, for some $U \in \gU(2)$, which in particular shows that 
$\Phi$ verifies the assertion of Theorem  \ref{theorem:stormer}. 

\smallskip 
Having settled the bistochastic case, we now want to deduce the general one. 
Two possible strategies to achieve that are:  

1. Focus on maps $\Phi$  generating {\em extreme rays} of $\cP(\C^2)$, and conclude via the Krein--Milman theorem. 

2. Focus on maps $\Phi$  belonging to the {\em interior} of $\cP(\C^2)$, and conclude 
by passing to the closure.

The usual approach, starting with St{\o}rmer's proof, was to use the first strategy. 
We will choose the second one.  
For a one-stop reading experience,  we also present 
at the end of this note 
a self-contained proof following the traditional outline, but highly streamlined.

\section{A proof via Brouwer's theorem} \label{section:brouwer}
The crucial  observation is that maps belonging to the interior of 
$\cP$ are, in a sense, equivalent to bistochastic ones.  

\begin{proposition} \label{prop:general-to-unital-TP} 
Let $\Phi : \cM_n^{\sa} \to \cM_n^{\sa}$ be a linear map which belongs to the interior of 
$\cP$, the cone of positivity-preserving maps. 
Then there exist  positive-definite operators  $A,B$  
such that 
\begin{equation} \label{eq:eneral-to-unital-TP}
\tilde{\Phi} = \Phi_A\circ \Phi \circ \Phi_B
\end{equation}
 is bistochastic, where $\Phi_A, \Phi_B$ are defined by \eqref{eq:PhiM}. 
\end{proposition}  
Once the Proposition  is shown, Theorem \ref{theorem:stormer} 
readily follows. Indeed,  \eqref{eq:eneral-to-unital-TP} is equivalent to $\Phi= \Phi_{A^{-1}}\circ \tilde{\Phi}  \circ \Phi_{B^{-1}}$, and appealing to 
the already proved bistochastic case shows that 
$\Phi$ admits a representation of the form \eqref{eq:stormer}. 
Finally, there are no issues with passing to the closure since 
any term in \eqref{eq:stormer} must belong to the compact set $\{\Psi \in \cP : \Phi - \Psi \in \cP\}$. 

{Proposition \ref{prop:general-to-unital-TP} is closely related to Theorem 4.7 from \cite{Gurvits04}.  
However,  \cite{Gurvits04} required a constructive -- and hence a relatively involved -- proof. 
Similar statements were known earlier for completely positive maps (see, e.g.,  \cite{GGHE08} and its references), 
but of course that would not be useful for our purposes.}  We thank David Reeb for bringing these references to our
attention. 

For our proof of  Proposition \ref{prop:general-to-unital-TP} we will need some notation and two lemmas. 
First, let $\Phi^*$ denote the usual functional analytic adjoint of $\Phi  : \cM_n^{\sa} \to \cM_n^{\sa}$, 
based on identifying $\cM_n^{sa}$ 
with its dual via  $\scalar{\rho}{\sigma}_{\HS} := \tr ({\rho}{\sigma})$.  
The following properties of the operation $^*$  are well-known (and easy to show).

\begin{lemma} \label{lemma:star-properties} 
Let $\Phi  : \cM_n^{\sa} \to \cM_n^{\sa}$. Then \\
(i) $\Phi \in \cP(\C^n)$ if and only if  $\Phi^* \in \cP(\C^n)$ \\ 
(ii) $\Phi$  is unital  if and only if  $\Phi^*$ is trace-preserving, and vice versa \\
(iii) if $M\in \cM_n$, then $\Phi_M^* = \Phi_{M^\dagger}$. 
\end{lemma}

The second lemma describes maps belonging to the interior of $\cP(\C^n)$;  
its proof is 
straightforward and based on very general principles.  

\begin{lemma} \label{lemma:interior}
Let $\Phi : \cM_n^{\sa} \to \cM_n^{\sa}$ be a linear map. The following conditions are equivalent.\\
(i) $\Phi$ belongs to the interior of $\cP(\C^n)$. \\
(ii) $\Phi = (1-t) \Psi + t\Omega$ with $t\in (0,1]$ and $\Psi  \in \cP(\C^n)$, where  
$\Omega(\rho) := (\tr \rho)\rho_* = (\tr \rho)~\Id/n$ 
is the completely depolarizing map.  \\
(iii) $\Phi^*$ belongs to the interior of $\cP(\C^n)$.  \\
(iv) If $\rho \in \cD(\C^n)$, then $\Phi(\rho)$ is positive definite. \\
(v) If $\rho \in \PSD(\C^n)$ and  $\rho \neq 0$, then $\Phi(\rho)$ is positive definite.
\end{lemma} 

\begin{proof} [Proof of Proposition \ref{prop:general-to-unital-TP}]  
Given positive definite $A$ and $B$,  let $\tilde{\Phi}$ be given by the formula from the Proposition. Then 
\begin{equation}   \label{eq:general-to-unital-TP1} 
\tilde{\Phi} \ \hbox{ is unital } \Leftrightarrow A \Phi(B^2)A = \Id \Leftrightarrow  \Phi(B^2)=A^{-2} 
\Leftrightarrow \Phi(B^2)^{-1} = A^2.
\end{equation} 
We next note that, by Lemma \ref{lemma:star-properties} (iii), Eq. \eqref{eq:eneral-to-unital-TP} 
can be rewritten as 
$\tilde{\Phi}^* = \Phi_B\circ \Phi^*\circ \Phi_A$. Accordingly,    
by Lemma \ref{lemma:star-properties} (ii), 
\begin{equation} \label{eq:general-to-unital-TP2}
\tilde{\Phi} \ \hbox{ is trace-preserving } \Leftrightarrow \tilde{\Phi}^* \hbox{ is unital } \Leftrightarrow  B \Phi^*(A^2)B = \Id \Leftrightarrow  \Phi^*(A^2)=B^{-2}. 
\end{equation} 
Solving the last equation in \eqref{eq:general-to-unital-TP2} for $B^2$ and substituting in 
\eqref{eq:general-to-unital-TP1} we are led to a system of equations
\begin{equation} \label{eq:general-to-unital-TP3}
B^2 = \Phi^*(A^2)^{-1} \ \hbox{ and } \  \Phi\big(\Phi^*(A^2)^{-1}\big)^{-1} = A^2. 
\end{equation} 
The second equation in  \eqref{eq:general-to-unital-TP3}  says that $S=A^2$ 
is a fixed point of the function 
\begin{equation} \label{eq:general-to-unital-TP4}
S \mapsto f(S) := \Phi\big(\Phi^*(S)^{-1}\big)^{-1} . 
\end{equation} 
Conversely, if  $S$ is a positive definite fixed point of $f$, then $A= S^{1/2}$ and 
$B = \Phi^*(A^2)^{-1/2}$ 
satisfy \eqref{eq:general-to-unital-TP1} and  \eqref{eq:general-to-unital-TP2} and 
yield $\tilde{\Phi}$ which is unital and trace-preserving. 
Note that, by Lemma \ref{lemma:interior}, the hypothesis 
``$\Phi$ belongs to the interior of $\cP$'' guarantees that all the inverses 
and negative powers above make sense,  
and that $f$ is well-defined and continuous on $\PSD \setminus \{0\}$. 

To find a fixed point of $f$ we want to use Brouwer's fixed-point theorem, which requires a 
(continuous) function 
that is a self-map of a compact convex set. One way to arrive at that setting is to consider 
$f_1 : \cD(\C^n) \to \cD(\C^n)$ defined by 
\begin{equation} \label{eq:general-to-unital-TP5}
f_1(\sigma) = \frac{f(\sigma)}{\tr f(\sigma)} . 
\end{equation} 
It then follows that there is $\sigma_0 \in D(\C^n)$ such that $f_1(\sigma_0) = \sigma_0$ 
and hence $f(\sigma_0) = \alpha \sigma_0$, where $\alpha = \tr f(\sigma_0)>0$. 
The final step is to note that  if we choose -- as before -- 
$A= \sigma_0^{1/2}$ and $B = \Phi^*(A^2)^{-1/2}$, then the corresponding 
$\tilde{\Phi}$ is trace-preserving and satisfies $\tilde{\Phi} (\Id) = \alpha ^{-1} \Id$, 
which is only possible if $\alpha =1$. In other words, $\sigma_0$ is a fixed point of $f$ 
that we needed to conclude the argument.  
%
\end{proof}
\begin{remark} If properly stated, Proposition \ref{prop:general-to-unital-TP} generalizes -- 
with essentially the same proof -- to maps $\Phi : \cM_m^{\sa} \to \cM_n^{\sa}$ with $m\neq n$. 
The correct conditions are that $\tilde{\Phi}$ is trace preserving and that it sends  the maximally mixed state 
$\Id/m \in \cD(\C^m)$ to the maximally mixed state 
$\Id/n \in \cD(\C^n)$. 
This suggests in particular a possible path to a simple proof of Woronowicz's theorem \cite{Woronowicz76} 
(a version of Theorem \ref{theorem:stormer} for maps $\Phi : \cM_2^{sa} \to \cM_3^{sa}$) 
by reducing it to the case of maps verifying these two conditions. 

\end{remark}

\section{A traditional proof} 

The second proof we present is based on the more traditional strategy, a description of the maps 
$\Phi$  generating {\em extreme rays} of $\cP(\C^2)$. 
That description is most conveniently expressed as a statement about the Lorenz cone 
$$
\mL_m = \big\{x=(x_0,x_1,\ldots,x_{m-1}) : x_0\geq 0, \; q(x) \geq 0 \big\}, 
$$ 
where $q(x) :=  x_0^2 - \sum_{k=1}^{m-1} x_k^2$. If  $\LP$ 
is the cone of linear maps on $\R^m$ that preserve $\mL_m$, we have \cite{LoewySchneider75}

\begin{proposition} \label{prop:loewy-schneider}
Let $\Phi : \R^m  \to \R^m $ be a 
linear map which generates an extreme ray of $\LP$. 
Then either $\Phi$ is an automorphism of $\mL_m $ or $\Phi$ is of rank one, in which case  
$\Phi = \ketbra{u}{v}$  for some $u, v \in \partial \mL_m  \setminus\{0\}$. 
If $m>2$, the converse implication also holds. 
\end{proposition}

Since $\PSD(\C^2)$   
is isomorphic to $\mL_4$, Proposition \ref{prop:loewy-schneider} yields a characterization of 
extreme rays of $\cP(\C^2)$ and -- by the Krein--Milman theorem -- reduces  the proof of 
Theorem \ref{theorem:stormer} to showing that the corresponding extreme maps admit a 
representation of type \eqref{eq:stormer}.

To establish the last fact, we note that the structure of the set of automorphisms of $\mathcal{L}_m$ 
is very well understood: they are of the form $t \Phi$, where $t>0$ and  $\Phi \in \gO^+(1,m-1)$, 
the {\em orthochronous subgroup} of the {\em Lorentz group} $\gO(1,m-1)$ 
of transformations preserving the 
quadratic form $q(x)= x_0^2-\sum_{k=1}^{m-1} x_k^2$.  However, for $m=4$ and for our purposes,  
it is more convenient to use the fact that automorphisms of $\PSD(\C^n)$ are of the form 
$\rho \mapsto V\rho V^\dagger$ or $\rho \mapsto V\rho^TV^\dagger$ for some $V\in \gGL(n)$, 
which immediately yields a representation of type \eqref{eq:stormer}. 
(This is an instance of Kadison's theorem \cite{Kadison65}, 
which  for $n=2$ is elementary and very simple.)

The case of rank one maps is even simpler:  
every element of $\partial \PSD(\C^2)\setminus\{0\}$ is of the 
form $\ketbra{\vphi}{\vphi}$, $\vphi \in \C^2\setminus\{0\}$, and so   $\Phi$  can be represented as
\[
\Phi(\rho) =   \tr(\rho \ketbra{\xi}{\xi}) \ketbra{\psi}{\psi} = \ket \psi \bra\xi \, \rho \ket \xi \bra \psi .  
\]
In other words, $\Phi = \Phi_{ \ketbra{\psi}{\xi}}$, as needed. 
It should be noted, however, that -- in absence further refinement -- 
this argument involves later an application 
of Carath\'eodory's theorem in a $15$-dimensional space (say, in $\{\Phi :  \tr \big(\Phi(\Id)\big) =n \}$), 
leading to a bound of $16$ on the number of terms in \eqref{eq:stormer}.

The above scheme of the proof of St{\o}rmer's theorem was apparently folklore for some time; 
it appears explicitly in \cite{MillerOlkiewicz15}. However, its value was limited by the fact that 
the proof  of Proposition \ref{prop:loewy-schneider} given in  \cite{LoewySchneider75} 
was itself long and computational. 
Our contribution, if any, consists in streamlining of the argument  given in \cite{Hildebrand05, Hildebrand07}, 
which rediscovered Proposition \ref{prop:loewy-schneider} and noted its relevance to 
the entanglement theory. 
The proof is based on the so-called $S$-lemma \cite{Yakubovich73}, 
 a well-known fact from control theory and quadratic/semi-definite programming.   

\begin{lemma} [$S$-lemma] \label{lemma:S-lemma-II}
Let $F, G$ be $n\times n$  symmetric real matrices. Assume that 
 there is an $\bar{x} \in \R^n$ such that $\bra{\bar{x}}G\ket{\bar{x}} > 0$.  
Then the following two statements about such $F, G$ are equivalent:\\
(i) if $x \in \R^n$ verifies $ \bra{x} G \ket{x}\geq 0$, then $ \bra{x} F \ket{x}\geq 0$\\ 
(ii) there exists  $\mu \geq 0$ such that $F - \mu G$ is positive semi-definite. 
\end{lemma}

We postpone the proof of the Lemma until the end of this Appendix and show how 
it implies the Proposition. (We leave out the ``converse'' part, which is easier 
and not needed for our purposes.)

\begin{proof} [Proof of Proposition \ref{prop:loewy-schneider}]  
The  case $\rank \Phi = 1$ is an immediate consequence of the following elementary observation, 
which completely characterizes extreme rays generated by rank one maps 
in a very general setting (we only need the ``only if'' part, which is very easy).  

\begin{lemma} \label{lemma:extreme-rank-one}
Let $\mC \subset \R^n$ be a nondegenerate cone and let $\cP(\mC)$ be the cone of 
linear maps preserving $\mC$.   
A rank one map $\Phi : \R^n \to \R^n$ generates an extreme ray of $\cP(\mC)$ iff 
it is of the form $\Phi = \ketbra{u}{v}$, with $u$ and $v$ generating extreme rays of 
respectively $\mC$ and the dual cone $\mC^*$.
\end{lemma}
Above, $\mC$ being {\em nondegenerate} means that $\dim \mC = n$ and $-\mC \cap \mC = \{0\}$, 
while the {\em dual cone} is defined by $\mC^* := \{x \in \R^n :  \braket{x}{y} \geq 0 \ \hbox{ for all } y \in \mC\}$. 

Next, assume that $\rank \Phi \geq 2$. Let $J \in \gO(n)$ be the diagonal matrix with 
diagonal entries $1,-1,\ldots,-1$; then $\bra{x} J \ket{x}  = x_0^2-\sum_{k=1}^{m-1} x_k^2 = q(x) $ for $x\in \R^n$.  
The map $\Phi$ preserving $\mathcal{L}_n$ (and hence $-\mathcal{L}_n$) means that 
the hypothesis (i) of Lemma \ref{lemma:S-lemma-II} is satisfied with $G=J$ 
and $F=\Phi^* J \Phi$. Since clearly $-J$ is not positive definite, it follows that 
there is $\mu\geq 0$ and a positive semi-definite operator $Q$ such that 
\begin{equation} \label{eq:Slemma-applied}
\Phi^* J \Phi = \mu J +Q .
\end{equation} 
We now notice that  since $\rank \Phi \geq 2$, 
there is $y = \Phi x \neq 0$ such that $y_0=0$. In particular,  
$\bra{x}\Phi^* J \Phi\ket{x} = \bra{y} J \ket{y} < 0$. Given that  $\bra{x} Q \ket{x} \geq 0$, 
it follows that $\mu$ can not be $0$.  
Next, if $Q=0$, \eqref{eq:Slemma-applied} means precisely that $\mu^{1/2}\Phi \in 
\gO(1,n-1)$ and so $\Phi$ is an automorphism of $\mathcal{L}_n$. 

To complete the argument, we will show that if $Q\neq 0$, then there is a rank one operator 
$\Delta$ such that $\Phi \pm \Delta \in \LP$. 
Since $\Phi$ and $\Delta$ have different ranks, 
they are not proportional. Hence $\Phi + \Delta$ and $\Phi -\Delta$ do not belong 
to the ray generated by $\Phi$, which implies that the ray is not extreme. 

Let $\ketbra{v}{v}$, $v\neq 0$,  be one of the terms appearing in the spectral decomposition of $Q$; 
then  $Q = Q' + \ketbra{v}{v}$, where $Q'$ is positive semi-definite. 
Next,  let  $u\in \R^n\setminus \{0\}$ be such that $\Phi^* J u = \delta v$, where $\delta$ is either $1$ or $0$. 
Such $u$ exists: if $\Phi^*$ is invertible, then  $u=J(\Phi^*)^{-1}v$  satisfies $\Phi^* J u = v$, 
while in the opposite case the nullspace of $\Phi^* J$ is nontrivial. 
We will show that, for  some $\e >0$, 
\begin{equation} \label{eq:Delta} 
\Phi + s \ketbra{u}{v} \in \LP \ \hbox{ if } \ |s| \leq \e ,
\end{equation} 
thus supplying the needed $\Delta = \e \ketbra{u}{v}$.  We have, by \eqref{eq:Slemma-applied} and by the choice of $u$, 
\begin{eqnarray} \label{eq:perturbation}
(\Phi + s\ketbra{u}{v})^* J (\Phi + s\ketbra{u}{v}) &=& \mu J +Q+ 2s\delta \ketbra{v}{v} +s^2 \ketbra{v}{u} J \ketbra{u}{v} 
\nonumber \\
&=& \mu J +Q'+(1+2s\delta +s^2 \bra{u}J\ket{u}) \ketbra{v}{v} .
\end{eqnarray}  
Since clearly $1+2s\delta +s^2 \bra{u}J\ket{u} \geq 0$ if $|s|$ is sufficiently small, it 
follows that, for such $s$,  $(\Phi + s\ketbra{u}{v})^* J (\Phi + s\ketbra{u}{v}) - \mu J$ is positive semidefinite. 
Thus we can deduce from the easy part of Lemma \ref{lemma:S-lemma-II} that $\Phi + s\ketbra{u}{v} \in \LP$, as needed. 
(To be precise, we need to exclude the possibility that $\Phi + s\ketbra{u}{v} \in -\LP$, 
but this is simple.)  
%
\end{proof}
It remains to prove Lemma \ref{lemma:S-lemma-II}.  {We follow \cite{PolikTerlaky07};   
we first restate the Lemma in a simpler~form~\cite{Yuan90}.}
 
\begin{lemma}[$S$-lemma reformulated] \label{lemma:S-lemma}
Let $M, N$ be $n \times n$ symmetric real matrices. The following two statements are equivalent:\\
(i) $\{ x \in \R^n \st \bra{x} M \ket{x} \geq 0 \} \cup \{ x \in \R^n \st \bra{x} N \ket{x} \geq 0 \} = \R^n$ 
\\ 
(ii) there exists  $t \in [0,1]$ such that the matrix $(1-t)M+t N$ is positive semi-definite.
\end{lemma} 

Lemma \ref{lemma:S-lemma-II} is an easy consequence of Lemma \ref{lemma:S-lemma} applied
with $M=F$ and $N=-G$.

\begin{proof} [Proof of Lemma \ref{lemma:S-lemma}] 
The implication $(ii) \Rightarrow (i)$ is straightforward. 
To show that $(i) \Rightarrow (ii)$, 
we argue by contradiction. 
Denote $M_t = (1-t)M+tN$ and assume that, for every $t \in [0,1]$, the smallest eigenvalue 
$\lambda_t$ of $M_t$ is strictly negative. Note that $t \mapsto \lambda_t$ is continuous. 
For $t\in [0,1]$, set 
\[
\Lambda_t := \{ x \in S^{n-1} : M_tx = \lambda_t x\} \neq \emptyset . 
\]
Then  $t \mapsto \Lambda_t$ upper semicontinuous in the sense that $t_n \to t$, 
$x_n \in \Lambda_{t_n}$ and $x_n \to x$ imply $x\in \Lambda_t$. 

Consider the sets $A = \{ x \in \R^n \st \bra{x} M \ket{x} \geq 0\}$ and 
$B = \{ x \in \R^n \st \bra{x} N \ket{x} \geq 0 \}$. We have $A \cup B = \R^n$ by hypothesis.
Since $M_0=M$, it follows that $\Lambda_0 \cap A=\emptyset$ and so  $\Lambda_0 \subset B$. 
Similarly, $\Lambda_1 \subset A$. 
Set 
\[
\tau = \sup \{t \in [0,1] : \Lambda_t \cap B \neq \emptyset \}. 
\]
We now note that $\Lambda_\tau  \cap B \neq \emptyset$; this is immediate if  $\tau=0$ and follows from 
upper semicontinuity of $t \mapsto \Lambda_t$ if $\tau >0$. 
For essentially the same reasons, $\Lambda_\tau  \cap A \neq \emptyset$. 

We now claim that $\Lambda_\tau  \cap A \cap B \neq \emptyset$. This is clear if the eigenvalue $\lambda_\tau$ 
is simple (note that all three sets, $\Lambda_\tau,  A$ and $B$, are symmetric by definition). On the other hand, 
if the multiplicity of 
$\lambda_t$ equals $k>1$, then $\Lambda_\tau$ is a $(k-1)$-dimensional sphere and hence is connected. 
Consequently, the closed nonempty sets $\Lambda_\tau  \cap A$ and  $\Lambda_\tau  \cap B$, the union 
of which is $\Lambda_\tau$,  
must have a nonempty intersection. 

To conclude the argument, choose $x \in \Lambda_\tau  \cap A \cap B \neq \emptyset$. 
Then, since $x \in  \Lambda_\tau$,
\[
\bra{x} M_\tau \ket{x} = \lambda_t < 0 .
\] 
On the other hand, since $x \in A \cap B$, 
\[
\bra{x} M_\tau \ket{x} = (1-\tau)\bra{x} M \ket{x} + \tau \bra{x} N \ket{x} \geq 0, 
\]
a contradiction.  
\end{proof}


\begin{thebibliography}{10} 
\footnotesize 
\bibliographystyle{alpha} 

\bibitem{book}
G. Aubrun and S. Szarek, 
 {\em Alice and Bob meet Banach. The Interface of Asymptotic Geometric
  Analysis and Quantum Information Theory.}  Book in preparation. 
  
\bibitem{GGHE08}  
  O. Gittsovich, O. G\"uhne, P. Hyllus, and J. Eisert. 
  {\em Unifying several separability conditions using the covariance matrix criterion.} 
  Phys. Rev. A 78, 052319 (2008).
  
\bibitem{Gurvits04}  
L. Gurvits,
{\em Classical complexity and quantum entanglement.}  
J. Comput. System Sci. 69 (2004), 448-484. 
  
\bibitem{Hildebrand05}
R. Hildebrand, 
{\em Cones of ball-ball separable elements.}  
arXiv.org eprint,  quant-ph/0503194 (2005).

\bibitem{Hildebrand07} 
R. Hildebrand, 
{\em Positive maps of second-order cones.}
Linear Multilinear Algebra 55(6) (2007), 575-597. 

\bibitem{HHH96}
M. Horodecki, P. Horodecki, R. Horodecki,
{\em Separability of Mixed States: Necessary and Sufficient Conditions.}
Physics Letters A 223, 1-8 (1996).

\bibitem{Kadison65} 
R. V. Kadison, 
{\em Transformations of states in operator theory and dynamics.}
Topology  3 (suppl. 2) (1965), 177-198. 

\bibitem{KSW09}
D.S. Kaliuzhnyi-Verbovetskyi, I.M. Spitkovsky and H.J. Woerdeman,
{\em Matrices with normal defect one.}
Oper. Matrices 3 (2009), no. 3, 401-438. 

\bibitem{KCKL00}
B. Kraus, J. I. Cirac, S. Karnas and M. Lewenstein,
{\em Separability in $2 \times N$ composite quantum systems.}
Phys. Rev. A 61 (2000), 062302.

\bibitem{LMO06}
J. M. Leinaas, J. Myrheim and E. Ovrum,
{\em Geometrical aspects of entanglement.}
Phys. Rev. A 74 (2006),~012313.

\bibitem{LoewySchneider75}
R. Loewy and H. Schneider, 
 {\em Positive operators on the $n$-dimensional ice cream cone.}
{ J. Math. Anal. Appl.} 49 (1975), 375-392. 

\bibitem{MillerOlkiewicz15} 
M. Miller and R. Olkiewicz, 
 {\em  Topology of the cone of positive maps on qubit systems.} 
J. Phys. A: Math. and Theor. 48(25) (2015), 255203.  

\bibitem{PolikTerlaky07} I. P\'olik, T. Terlaky,  
{\sl A Survey of the S-Lemma.} SIAM Rev. 49 (2007), 371-418.

\bibitem{Stormer63}
E. St{\o}rmer, 
{\em  Positive linear maps of operator algebras.} 
 {Acta Math.} 110 (1963) 233-278. 

\bibitem{VDM01}
F. Verstraete, J. Dehaene and B. DeMoor,
{\em Local filtering operations on two qubits.}
Phys. Rev. A 64 (2001), 010101(R).
 
 \bibitem{Woronowicz76}
S. L. Woronowicz, 
{\em Positive maps of low dimensional matrix algebras.} 
\newblock { Reports on Mathematical Physics} 10(2) (1976), 165-183. 

 \bibitem{Yakubovich73} 
 V. A. Yakubovich, 
 {\em Minimization of quadratic functionals under quadratic constraints and
the necessity of a frequency condition in the quadratic criterion for absolute stability of
nonlinear control systems.}  Soviet Math. Dokl., 14 (1973), pp. 593-597. 
 
 \bibitem{Yuan90} Y.-X. Yuan, 
 {\em On a subproblem of trust region algorithms for constrained optimization.} 
  Math. Programming, 47 (1990), pp. 53-63.


\end{thebibliography}
\end{document}